\documentclass[12pt]{article}
\usepackage[utf8]{inputenc}
\usepackage[IL2]{fontenc}
\usepackage[english]{babel}
\usepackage{amssymb}
\usepackage{amsmath} 
\usepackage{amsthm} 
\usepackage{amssymb} 
\usepackage{commath}
\usepackage{amsfonts} 
\usepackage{enumerate}
\usepackage{subcaption}
\usepackage{graphicx}
\usepackage{pgf,tikz}
\usetikzlibrary{arrows}
\usepackage{xcolor}

\newtheorem{theorem}{Theorem}

\providecommand{\keywords}[1]
{
  \small	
  \textbf{\textit{Keywords:}} #1
}

\providecommand{\MSC}[2]
{
  \small	
  \textbf{MSC:} #1
}

\newtheorem{lemma}[theorem]{Lemma}
\newtheorem{remark}[theorem]{Remark}
\newtheorem{corollary}[theorem]{Corollary}
\theoremstyle{definition}

\def\M{\mathcal{M}}
\def\J{\mathcal{J}}
\def\N{\mathbb{N}}

\def\G{\mathcal{G}}

\def\inte{\mathrm{int}}
\def\fr{\mathrm{fr}}
\title{On distributional spectrum of piecewise monotonic maps}
\date{December 2020}
\author{Vojt\v ech Pravec, Jan Tesar\v c\' ik}

\begin{document}
\renewcommand{\proofname}{Proof}
\renewcommand{\figurename}{Figure}
\title{On distributional spectrum of piecewise monotonic maps}

\maketitle
\begin{abstract}
We study a certain class of piecewise monotonic maps of interval. 
These maps are strictly monotone on finite interval partition, 
satisfies Markov condition and have generator property. We show 
that for a function from this class distributional chaos is always 
present and we study its basic properties.
 
Main result states that distributional spectrum as well as weak 
spectrum is always finite. This is a generalization of similar 
result for continuous maps on the interval, circle and tree.
 
Example showing that conditions on mentions class can not be 
weakened is presented. 
\end{abstract}
\\
\MSC  {37B40, 37E05}\\

\section{Introduction}

Distributional chaos and structure of spectrum for continuous maps 
of the one-dimensional spaces as the interval, the circle or tree 
is well known, see \cite{ScSm1,MM1,T1} for details. The class of piecewise monotonic 
not necessary continuous interval maps is a natural generalisation 
of these maps. In papers \cite{H2,H3,H1,HR1,R1,R2} F.~Hofbauer and P.~Raith provided basic tools 
and properties of this class of dynamical systems and showed that 
in general dynamics of map with discontinuity points can have 
significantly different behaviour \cite{HRS}. Therefore we focus on piecewise 
monotonic functions which fulfil Markov conditions  and condition 
of generators. In this case we have obtained that spectrum of distributional 
functions is finite. Spectrum distributional contains minimal incomparable lower 
distributional functions. Just integral of difference between upper and lower distributional 
functions is considered as a measure of chaos \cite{ScSm1,ScSm2}.

We also provide an example which shows that without condition of generators 
piecewise monotonic map can have infinite spectrum. This also can happed when 
we consider continuous map on more general spaces like dendroids~\cite{T1}.  \\

Let $I$ be an interval. For any subset $K\subset I$ let $\inte (K)$ be its \textit{interior}, $\fr (K)$ be its \textit{border} and $K^c$ be its \textit{complement}.

A map $f:I\rightarrow I$ is called \textit{piecewise monotone} if there is a finite colection $\mathcal{J}$ of pairwise disjoint nondegenerate intervals $\{J_1, \dots , J_N  \} $ such that $\left.f\right|_{ \inte (J_i)}$ is continuous and strictly monotone, and $f$ is semicontinuous at $c\in \fr(J_i)$. For any $n\geq 0$, $f^n$ denotes the \textit{n-th iteration} of $f$, it is defined recursively by $f^0=id_I$ and $f^n=f\circ f^{n-1}$. 

Let $C_0(f)$ be the set of endpoints of $J_i \in \J$, $C_0(f)= \{c| c\in \fr (J_i) \}$, and we called this set as a set of critical points. For nonegative integer $n$ let $C_n(f)$ be the set of points which will map to the set of critical points in $n$ iterations, or less, $C_n(f)=\{x\in I| f^i(x)\in C_0(f) \mbox{ for some } 0\leq i \leq n \}$, and let $C(f)$ be the limit case of $C_n(f)$, $C(f)=\{x\in I| f^i(x) \in C_0 \mbox{ for some } i\in \N \}$.
The \textit{trajectory} of a point $x\in I$ is the sequence $\{f^i(x)\}_{i=0}^\infty$ and the \textit{orbit} of a point $x\in X$ is the set $\{x,f(x),f^2(x),\dots\}$. The \textit{$\omega$-limit set} of a point $x\in I$ is a set of accumulation point of trajectory of $x$ and it is denote by $\omega_f(x)$. \\

Based on work of A.~N.Sharkovsky in \cite{Sh2,Sh3}, maximal $\omega$-limit were characterised 
and their property widely studied, see e.g. \cite{Bl95,MM2}. Result in the case of continuous 
maps are based on properties of these $\omega$-limit sets. We adopt this characterisation, but 
we drop assumption of maximality because in the case of discontinuous maps maximal 
$\omega$-limit need not exist \cite{HRS}. If an $\omega$-limit set is finite then it is 
a \textit{cycle}. If it is infinite and does not contain periodic point the is called 
a \textit{solenoid}. The $\omega$-limit set is called a \textit{basic} set if it is infinite 
and contains a periodic point. The last case of $\omega$-limit sets play crucial role in 
distributional chaos.\\

We say that $\omega_f(x)$ is \textit{maximal} if there is no $y\in I$ such that $\omega_f(x)\subsetneq \omega_f(y)$. We denote $\Omega$ as the set of maximal basic $\omega$-limit sets.\\

In \cite{ScSm1}, Schweizer and Smítal have defined the notion of distributional chaos. The definition is based on distributional function of pairs of points. For any $x,y\in I$, a positive integer $n$ and any real $t$, let 
\begin{equation}\label{(1)}
\xi(x,y,t,n)=\#\{i;0\leq i\leq n-1\text{ and }\delta_{xy}(i)<t\}
\end{equation}
where $\#S$ is the cardinality of the set $S$ and $\delta_{xy}(i):=\abs{f^i(x)-f^i(y)}$ is the distance of $i$-th iteration of points $x$ and $y$. Put
\begin{equation}\label{(2)}
F_{xy}^\ast(t)=\limsup_{n\rightarrow\infty}\frac{1}{n}\xi(x,y,t,n),
\end{equation}
\begin{equation}\label{(3)}
F_{xy}(t)=\liminf_{n\rightarrow\infty}\frac{1}{n}\xi(x,y,t,n).
\end{equation}

Function $F_{xy}^\ast$ is called an \textit{upper distributional function} and $F_{xy}$ is called \textit{lower distributional function}. Both $F_{xy}^\ast$ and $F_{x,y}$ are nondecreasing functions with $F_{xy}^\ast(t)=F_{xy}(t)=0$ for every $t<0$ and $F_{xy}^\ast(t)=F_{xy}(t)=1$ for every $t>\abs{I}$.

Let $f$ be a map on interval $I$ and $x,y\in I$. Then the pair $(x,y)$ is \textit{isotectic} (with respect to $f$) if, for every positive integer $n$, the $\omega$-limit sets $\omega_{f^n}(x)$ and $\omega_{f^n}(y)$ are subsets of the same maximal $\omega$-limit set of $f^n$.

Equivalently, the pair $(x,y)$ is isotectic if $\omega$-limit sets $\omega_f(x)$ and $\omega_f(y)$ are subset of the same maximal $\omega$-limit set $\tilde{\omega}$ and if for any $f^{\ast}$-periodic set $J$ such that $J\cup f(J)\cup f^2(J)\cup\dots\cup f^{m-1}(J)\supset\tilde{\omega}$, where $m$ is the period of $J$, there is $j\geq 0$ such that both $f^j(x)$ and $f^j(y)$ belong to $J$. Note that this equivavalent definition was introduced in \cite{ScSm1} and such points are called \textit{synchronous}.

Put $Iso(f)=\{(x,y)\in I^2;(x,y)\text{ is isotectic}\}$ and $D(f)=\{F_{xy}; (x,y)\in Iso(f)\}$. Then the \textit{spectrum of $f$} is the set of minimal elementrs of $D(f)$, denote such set by $\Sigma(f)$.\\

From definition (see (\ref{(2)}), (\ref{(3)})) we get that for every $x,y \in I$ and positive integer $k$ $F_{f^k(x)f^k(y)}(t)=F_{xy}(t)$ and $F^\ast_{f^k(x)f^k(y)}(t)=F^\ast_{xy}(t)$. If we have points $x,y \in C(f)$, then there are positive integers $k,l$ such that $f^k(x)\in C_0(f)$ and $f^l(y)\in C_0(f)$, we take $m$ as maximum of these two positive integers and we know that $f^m(x),f^m(y)\in C_0(f)$, we label $u=f^m(x)\mbox{ and } v=f^m(y)$ and from previous we know that $F_{xy}(t)=F_{uv}(t)$ and $F^\ast_{xy}(t)=F^\ast_{uv}(t)$. In view of this fact in theorems and lemmas it is suficient to concider $C_0(f)$ instead of $C(f)$, when we deal with $F^\ast$, $F$ of pair of points. \\

From now on we additionally suppose that our piecewise monotonic map satisfy \textit{Markov condition} and has \textit{generator}, we denote class of such a map by $\M(I)$. We say that piecewise monotonic map $f$ satisfies \textit{Markov condition} if for any $c\in C_0$ one-sided limits are in $C_0$. The collection of intervals  $\J$ is  \textit{the generator} if for every sequence $\{k_i\}_{i=0}^\infty$ of natural numbers $k_i\in\{1,2,\dots,N\}$ the set $\bigcap_{j=0}^{\infty} f^{-j}(J_{k_j})$ contains at most one point.\\

We end this section with notations of needed properties. First of them is \textit{property of irreducibility}, which describe subset of generator which is in some sense maximal and any member in this subset can access any member in this subset. Second of them replaces classic definition of periodic set, which does not work very well in the case of discontinous map.\\

We say that set $\G \subseteq \J$ has \textit{property of irreducibility} if for all pair of sets $J_i, J_j$ from $\G$ there is a positive integer $n_{i,j}$ such that $f^{n_{i,j}}(\inte (J_i))$ covers $\inte (J_j)$ and if for all pairs of sets $J_i, J_k$, where $J_i$ is from $\G$ and $J_k$ is from $\J \setminus \G$, all iterations of $\inte (J_k)$ never intersects $\inte (J_i)$ or all iterations of $\inte (J_i)$ never intersects $\inte (J_k)$.
Let $A$ be a set with $\inte(A)\neq \emptyset$, we say that $A$ is a \textit{$f^{\ast}$-periodic set} if there is a $m\geq 1$ with $\inte(f^m(A))=\inte (A)$.
For $\mathcal{L} \subseteq \J$  we define the \textit{set of all admissible sequences} $S(f,\mathcal{L})= \{ \{\alpha_i\}| J_{\alpha_i}\in \mathcal{L} \mbox{ and } f(J_{\alpha_i})\cap J_{\alpha_{i+1}}\neq \emptyset \}$.\\

The following theorem is our main result. Its proof can be found in section 4.

\begin{theorem}
Let $f \in \M(I)$, then both the Spectrum $\Sigma(f) $ and the Weak Spectrum $\Sigma_w(f)$ are nonempty and finite.
\end{theorem}

\section{Properties of $\omega$-limit set}

\begin{lemma}\label{IL}
Let $f\in PMM(I)$ with Markov condition and let $u,v \in I$. Let $\{ U_i \}_{i=1}^{\infty} $, $\{ V_i \}_{i=1}^{\infty}$ be compact intervals with 
$\lim_{i\rightarrow \infty} U_i = u$, $\lim_{i\rightarrow \infty} V_i = v$, and such that for all positive integers $i,j $ there are positive integers $u(i,j)$ and $v(i,j)$ with $  V_j \subset f^{u(i,j)}(U_i) $ and $ U_j \subset f^{v(i,j)}(V_i) $. Then $\{ u,v \} \subset \omega(y)$ for some $y\in I$.
\end{lemma}

\begin{proof}
Define a decreasing sequence $\{J_i \}_{i=1}^{\infty}$ of compact intervals and increasing sequence $\{ n(i) \}_{i=1}^{\infty}$ of positive integers as follows: $J_1 = U_1$ and $n(1)=u(1,2)$. Then $f^{n(1)}(J_1) \supset V_2 $, choose $J_2 \subset J_1$ such that $f^{n(1)}(J_2)=V_2$. Take $n(2)=n(1)+v(2,3)$. Then $f^{n(2)}(J_2)= f^{v(2,3)}(V_2) \supset U_3$ and there is $J_3 \subset J_2$ such that $f^{n(2)}(J_3)=U_3$. Then there are $J_4$ and $n(3)$ such that $f^{n(3)}(J_4)=V_4$, etc. Let $y\in \bigcap_{i=1}^{\infty} J_i$. Since the trajectory of $y$ visits every neighborhood of $u$ and every neighborhood of $v$, the results follows.
\end{proof}

\begin{lemma}\label{L4}
Let $f\in \M(I)$ and let $\J$ be its generator. Let there is a sequence of integers $\{ \beta_i \}_{i=0}^{\infty} $ such that $\bigcap_{i=0}^{\infty}f^{-i}(J_{\beta_i})\neq \emptyset $, then\\
 $ \bigcap_{i=0}^{\infty}\overline{f^{-i}(J_{\beta_i})} \setminus \bigcap_{i=0}^{\infty}f^{-i}(J_{\beta_i}) \subset C(f)   $ .
\end{lemma}
\begin{proof}
Note that $ \overline{f^{-i}(J_\alpha)}=f^{-i}(J_\alpha) \cup C^{i,\alpha}  $, where $C^{i, \alpha} = \{ c_1, c_2| f^i(c_1), f^i(c_2) \in \fr (J_{\alpha})  \} \subset C(f)$. 
Therefore it is easy to see that $\bigcap_{i=0}^{\infty}\overline{f^{-i}(J_{\beta_i})} \setminus \bigcap_{i=0}^{\infty}f^{-i}(J_{\beta_i}) $ is either empty, or is equal to $ \{ c \} $, where $c\in C(f)$.
\end{proof}

The following lemma is probably the most important lemma of this section, it show us a strong connection between maximal $\omega$-limit set and subset of generator
with property of irreducibility and we use this lemma to prove lemmas which follows.

\begin{lemma}\label{FG}
Let $f\in \M(I)$ and let $\J$ be its generator. There is a $\G \subset \J$ such that $\G$ has property of irreducibility if and only if there 
is $x\in I $ with maximal basic $\omega$-limit set $\omega(x)\subset \bigcup_{J_i \in \G} \overline{J_i}$.
\end{lemma}
\begin{proof}
Let $\G \subset \J$ have property of irreducibility. For showing that there is a $x\in I$ with maximal basic $\omega$-limit set $\omega(x)\subset \bigcup_{J_i \in \G} \overline{J_i}$ we follow the well known procedure from Symbolic dynamic (for details see Theorem 2 in \cite{H1}), where we construct infinite sequence of indices of intervals such that any finite admissible sequence from $S(f,\G)$ will be there and this constructed sequence $\{\beta_i\}_{i=0}^{\infty}$ is as well from $S(f,\G)$.\\
By Lemma~\ref{L4} and properties of generator we can see that the intersection $\bigcap_{i=0}^{\infty}f^{-i}(J_{\beta_i})$ contains 
exactly one point $x$. Trajectory of this point $x$ visits any neighborhood of any point $y$ for which his iterations will remain in $\G$.\\
To see $\omega(x)$ is maximal assume that there is some other maximal $\omega(z)$ on irreducible $\G$, then any point $y\in \bigcup_{J_i \in \G}\inte(J_i) $ for which his iterations will remain in $\G$ has to be in $\omega(z)$ and $\omega(x)\subset \omega(z)$. 
If there is some a point $p\in \omega(z)\setminus \omega(x)$, then $p\in \inte(J_k)$ where $J_k \in \J \setminus \G$ which is contradiction with irreducibility of $\G$.

Conversely, let there is $x\in \bigcup_{J_i \in \G} J_i $ with maximal basic $\omega$-limit set $\omega(x)$.\\
If there are some $J_k,J_l \in \G$ with $f^n(\inte(J_k))\cap \inte(J_l) =\emptyset$ for any positive integer $n$, then it would be impossible for iteration of the point $x$ to be infinitely many times in any neighbourhood of any point from intervals $J_k$ and $J_l$, therefore we get that for any two $J_k, J_l$ from $\G$ there has to be some positive integer $n$ with $f^n(\inte(J_k))\cap \inte(J_l) \neq \emptyset$. \\
If there is $J_k \in \J \setminus \G$ with positive integers $n,m$ such that $f^n(\inte(J_k))\cap \inte(J_i) \neq \emptyset$ and 
$f^m(\inte(J_j))\cap \inte(J_k)\neq \emptyset $ for some $J_i, J_j \in \G$, then $G\cup \{ J_k \}$ is a irreducible subset of generator and we can find $z$ such that $\omega(z)$ is maximal basic set with $\omega(z)\cap \inte(J_k)\neq \emptyset$ and $\omega(x)\subset \omega(z)$, wich contradict maximality of $\omega(x)$.
\end{proof}

\begin{remark}
If $f(\inte(J_j))\subset \bigcup_{J_i \in \G} \overline{J_i} $ hold for any $J_j \in \G$, then there is $x\in I$ such that $\omega(x)=\bigcup_{J_i \in \G} \overline{J_i}$.
\end{remark}

\begin{lemma}\label{BpG}
For any interval $J_p$ from the generator $\J$ there is at most one $\omega \in \Omega$ with $\inte(J_p)\cap \omega \neq \emptyset$.
\end{lemma}
\begin{proof}
Assume that $J_p\in \J$ and $\omega(x)$ is maximal basic $\omega$-limit set with $\omega(x)\cap \inte(J_p)\neq \emptyset$. From Lemma~\ref{FG} we can see that $J_p \in \G$ where $\G\subset \J$ has property of irreducibility. 
If there is some other maximal basic $\omega$-limit set $\omega(y)\neq \omega(x)$ with $\omega(y)\cap \inte(J_p)\neq \emptyset$, then by Lemma~\ref{FG} there is $\mathcal{H} \subset \J $ with property of irreducibility such that $\omega(y) \subset \bigcup_{J_j \in \mathcal{H}} \overline{J_j}$ and $J_p\in \mathcal{H}$. Therefore $\G \cup \mathcal{H}$ forms set with property of irreducibility and by Lemma~\ref{FG} there is some maximal basic $\omega$-limit set $\omega(z)$ such that $\omega(x)\subset \omega(z)$ which contradicts maximality of $\omega(x)$.
\end{proof}

Previous lemma tells us that the interior of a set from generator can intersect only with one maximal basic set, which immediately gives us a finite number of maximal basic sets and also that the intersection of two different maximal basic sets is always finite and under $C_0(f)$.

\begin{lemma}\label{PerP}
 Let $f\in \M(I)$ and $\omega \in \Omega$, then periodic points of $\omega$ form a dense subset of $\omega$.
 \end{lemma}
\begin{proof}
By Lemma~\ref{FG} there is $\G\subset \J$ such that $\G$ has property of irreducibility and $\omega\subset\bigcup_{J_i\in\G}\overline{J_i}$. Now, let $x\in\omega$ and let $U$ be its arbitrary neigborhood. Then, there is $V\subset U$ such that $V=\bigcap_{i=0}^jf^{-i}(J_{\alpha_i})$, where $J_{\alpha_i}\in\G$. Since $\G$ is irreducible, there is $l\geq j$ such that the finite sequence $\alpha_1,\alpha_2\dots,\alpha_j$ can be extended into infinite periodic sequence  $\alpha_1,\alpha_2\dots,\alpha_j,\dots,\alpha_l,\alpha_1,\dots$ in such way that $\bigcap_{i=0}^\infty f^{-i}(J_{\alpha_i})=y$. It is easy to see that $y$ is a periodic point and by Lemma~\ref{FG} $y\in \omega$ and since $x$ and $U$ was chosen arbitrary the lemma follows.
\end{proof}

\begin{lemma}\label{PerfB}
Let $f\in \M(I)$ and $\J=\{J_1, J_2, \dots, J_n\}$ be its generator. Let $\omega \in \Omega$, then $\omega$ is perfect set.  
\end{lemma}
\begin{proof}
By definition $\omega$ is a closed set. Now, assume that $x\in\omega$ is an isolated point. If $x $ would be a nonperiodic point then we get a contradiction with Lemma~\ref{PerP}. Therefore we may assume that $x$ is a periodic point and since $\omega$ is infinite $x$ is an attracting periodic point, but that is contradiction with assumption of existence of generator and Markov condition.
\end{proof}

\begin{lemma}\label{MaxBaz}
Let $f\in \M(I)$, let $x\in I$ such that $\omega(x) \subset \tilde{\omega}$, where $\tilde{\omega}\in \Omega $. Then there is a positive integer $N$ such that $f^n(x) \in \tilde{\omega}$ for any $n \geq N$.
\end{lemma}
\begin{proof}
From Lemma~\ref{FG} $\tilde{\omega}\subset \bigcup_{J_i \in \G} \overline{J_i} $ where $\G$ is irreducible subset of $\J$.\\
If $f^n(x)\notin \omega(x) $ for any positive integer $n$, then in $\omega(x)$ has to be some attracting periodic point, which is contradiction with existence of generator. Therefore there is a positive integer $N$ such that $f^N(x)\in \omega(x)\subset \tilde{\omega}$ and $f^N(x)\in J_i$ where $J_i \in \G$. 
If $f^N(x)\in J_i \cap C(f)$, then $\omega$-limit set of $f^N(x)$ is a cycle which cannot be attracting and therefore $f^{N+j}\in \omega(x)$ for any positive integer $j$. If $f^N(x)\in J_i \cap C(f)^c$ then from irreducibility of $\G$ we see that $f^{N+j}(x) \in J_{\alpha_j}$ where $J_{\alpha_j}\in \G$ for any $j$, therefore $f^{N+j}(x) \in \tilde{\omega}$, since $\tilde{\omega}$ contains any point of $J_{\alpha_j}\in \G$ whose itterations stays in $\bigcup_{J_{\alpha_i}\in \G}$ (see proof of Lemma~\ref{FG} for details).
\end{proof}

The rest of this section deals with results about properties of $f^{\ast}$-periodic set and are used to prove Lemma 14, which is irreplaceable for Main Theorem. In the rest of this section we assume that our function $f$ is from $\M(I)$.

\begin{lemma}\label{DenSet}
$C(f)$ is a dense subset of $I$.
\end{lemma}
\begin{proof}
Let $(a,b)\subset J_{k_0}$, where $J_{k_0}$ is a set from generator $\J$, then there is a point $p\in (a,b)$ and positive integer $j$ such that $f^{j}(p)\in C_0(f)$. Otherwise for any positive integer $n$ $f^n((a,b))\subset J_{k_n}$ where $J_{k_n}$ is a set from generator. Therefore $\bigcap_{i=0}^{\infty} f^{-i}(J_{k_i}) \supset (a,b) $ which is contradiction with generator.

\end{proof} 

\begin{lemma}\label{PerInt}
Let $\J=\{J_1,\dots , J_n\} $ be the generator. For any interval $(a,b)\subset J_i$, $J_i\in \J$, there is an interval $(c,d)\subset (a,b)$, a positive integer $k$ and a set $L\subset \{1,\dots,n \}$ such that $\inte (f^k ((c,d))) = \inte ( \bigcup_{l\in L} J_l ) $.
\end{lemma}
\begin{proof}
By Lemma~\ref{DenSet} there are points $c,d \in (a,b)$, $c<d$, and positive integers $j_c,j_d$ such that $f^{j_c}(c)\in C_0(f)$, $f^{j_d}(d)\in C_0(f)$ and 
$f^{\max\{j_c,j_d\} }(c)\neq f^{\max\{j_c,j_d\} }(d) $. The lemma follows from that. 
\end{proof}

\begin{corollary}\label{cor1}
If $K=(a,b)\subsetneq \inte(J_i) $, where $J_i\in \J$, then $K$ is not $f^{\ast}$-periodic set.
\end{corollary}

This corollary tells us that only possible $f^{\ast}$-periodic set are sets from generator, or some union of sets from generator. To see that there is always a $f^{\ast}$-periodic set it is sufficient to look at interior of iterations of whole set $I$. Let $\inte(f^i(I))= \inte ( \bigcup_{l\in L_i} J_l )$ where $L_i \subset \{ 1, \dots , n \} $ and $J_l \in \J$. Either $L_{i+1} \subsetneq L_i$, or $L_{i+1}=L_i$. Since we have finite number of set in generator, equality is inevitable. Any $f^{\ast}$-periodic set $U$ of period $m\geq 1$ is either $U=\inte(J_l)$ or $U=\inte(\cup_{l\in L} J_l )$, where $J_l \in \J$ and $L\subset \{1,\dots,n \}$.

\begin{corollary}\label{cor}
Let $\omega\in \Omega$. Then there is a minimal $f^{\ast}$-periodic set $U$ such that $\text{Orb}(U)\supset \omega$.
\end{corollary}

\begin{lemma}\label{PreStTr}
Let $U$ be a minimal $f^{\ast}$-periodic set (of period $m\geq 1$) such that $\text{Orb}(U)\supset \omega\in \Omega$. Then if $K_1,K_2$ are intervals such that $K_1\cap \omega$ is infinite and $K_2\subset \text{int}(U)$ then $f^{a+bm}(K_1)\supset K_2$ for some $a$ and for all sufficiently large $b$.
\end{lemma}
\begin{proof}
By Lemma~\ref{FG} $\omega\subset \bigcup_{J_i \in \G} \overline{J_i}$ where $\G$ is an irreducible subset of $\J$. We can choose $J_l\in\G$ such that $K_1\cap J_l \cap \omega$ is infinite, $\inte (J_l \cap K_1) $ is an interval under $J_l$, by Lemma~\ref{PerInt} we can find $k$ such that $f^k(\inte (J_l \cap K_1))\supset \inte (\bigcup_{J_g \in \G_1} J_g )  $, where $\G_1 \subset \G$. The proof follows from corollary~\ref{cor} and property of irreducibility.
\end{proof}

We will describe situation above by saying that $f|_{\omega}$ is \textit{strongly transitive} in $U$.\\

It is easy to see that we can choose intervals $K_1, K_2$ such that they form a horseshoe.
\begin{corollary}
Any piecewise monotone function $f$ with Markov condition and generator has positive topological entropy.
\end{corollary}

For more details see \cite{BC} (Chapter VIII., proposition 8).

\section{Distributional functions}

\begin{lemma}\label{Spf}
Let $f\in \M(I)$ and $x\in C(f)^c $. Then for any $n\in \mathbb{N}$ there is $\delta>0$ such that $\left.f\right|_{f^i(x-\delta,x+\delta})$ is continuous for $0\leq i\leq n$.
\end{lemma}
\begin{proof}
For any positive integer $n$ the set $C_n(f)$ is finite, therefore there is $\delta>0$ such that $(x-\delta,x+\delta)\cap C_n(f)=\emptyset$.
\end{proof}

The following lemmas are modified version of lemmas from \cite{ScSm1} for our case. From now on let $F_{xy}(t)$, $F^{\ast}_{xy}(t)$ and $\xi(x,y,k,t)$ be as is defined in the section Introduction.

\begin{lemma}\label{V51}
Let $f\in \M(I)$, $c\in C(f)$ and $x,y \in C(f)^c $. \\
Case a): Let $F_{xy}^\ast(t)$ and $F_{xy}(t)$ be continuous at $t$, then, for any $\varepsilon>0$, there are positive integers $k,q$, arbitrarily large, and $\delta>0$ such that
\begin{equation}
\frac{1}{k}\xi(u,v,k,t)<F_{xy}(t)+\varepsilon
\end{equation}
and
\begin{equation}
\frac{1}{q}\xi(u,v,q,t)>F^{\ast}_{xy}(t)-\varepsilon
\end{equation}
whenever $\vert x-u\vert<\delta$ and $\vert y-v\vert<\delta$\\
Case b): Let $F_{xc}^\ast(t)$ and $F_{xc}(t)$ be continuous at $t$, then for any $\varepsilon>0$, there are positive integers $k,q$, arbitrarily large, and $\delta>0$ such that
\begin{equation}
\frac{1}{k}\xi(u,c,k,t)<F_{xc}(t)+\varepsilon
\end{equation}
and
\begin{equation}
\frac{1}{q}\xi(u,c,q,t)>F^{\ast}_{xc}(t)-\varepsilon
\end{equation}
whenever $\vert x-u\vert<\delta$.
\end{lemma}

\begin{proof}
Case a) For $x,y\in C(f)^c$ and for given $\varepsilon$ let's choose $\varepsilon_1$ such that $F_{xy}(t+\varepsilon_1)<F_{xy}(t)+\varepsilon/2$ and $F^{\ast}_{xy}(t-\varepsilon_1)>F_{xy}(t)-\varepsilon/2$, existence of $\varepsilon_1$ follows from Lemma~\ref{Spf}. Let's observe that for $\varepsilon$ there is a positive integer $k$ such that $ \frac{1}{k} \xi(x,y,k,t+\varepsilon_1)< F_{xy}(t+\varepsilon_1)+\varepsilon/2$. The first inequality follows from the fact that if $\delta>0$ is sufficiently small, then $\xi(u,v,k,t)< \xi(x,y,k,t+\varepsilon_1)$ whenever $|x-u|<\delta$ and $|y-v|< \delta$. The proof of the second inequality is similar.\\
Proof of the Case b) is analogous.
\end{proof}

\begin{lemma}\label{V52}
Let $f\in \M(I)$, $c\in C_0(f)$ and let $\omega_1$ and $\omega_2$ be maximal basic sets.
Assume that there are $f^{\ast}$-periodic sets $U$, $V$ and countable sets $Q\subset I^2$ and $P\subset I$, where $Q$ is the set of pairs $(u,v)$ such that $u\in \omega_1 \cap \inte(U)\cap C(f)^c$ and $v\in \omega_2 \cap \inte(V)\cap C(f)^c$, and $P$ is the set of points $p\in \omega_1 \cap \inte(U)\cap C(f)^c$, and, furthermore, that $f|_{\omega_1}$ is strongly transitive in $\inte(U)$ and $f|_{\omega_2}$ is strongly transitive in $\inte(V)$.\\
Case a):Then there are points $x\in \omega_1\cap U$ and $y\in \omega_2\cap V$ such that for any $t>0:$
\begin{equation}\label{2d5.2/1}
F_{xy}(t)\leq \inf\{F_{uv}(t);(u,v)\in Q\}
\end{equation}
and
\begin{equation}\label{2d5.2/2}
F_{xy}^\ast(t)\geq \sup \{F_{uv}^\ast(t);(u,v)\in Q\}.
\end{equation}
Case b): Then there is a point $x\in \omega_1\cap U$ such that for any $t>0$
\begin{equation}\label{1d5.2/1}
F_{xc}(t)\leq \inf\{F_{uc}(t);u\in P\}
\end{equation}
and
\begin{equation}\label{1d5.2/2}
F_{xc}^\ast(t)\geq \sup \{F_{xc}^\ast(t);u\in P\}.
\end{equation}
\end{lemma}
\begin{proof}
Fist we proove Case a)\\
Let $T$ be a countable set, dense in I, and such that, for any $(u,v)\in Q$ and any $t\in T$, both $F_{uv}$ and $F_{uv}^\ast$ are continuous at t. Let $\{t_j\}_{j=1}^\infty$ and $\{u_j,v_j\}_{j=1}^\infty$ be a sequences of points from $T$ and $Q$, respectively, such that for any $t\in T$ and any pair $(u,v)\in Q$, $t=t_j$, $u=u_j$ and $v=v_j$ for infinitely many $j$.

Next, using induction, we define positive integers 
$$k_1<q_1<k_2<q_2\dots <k_i<q_i<\dots$$
and decreasing sequences $\{U_i\}_{i=1}^\infty$ and $\{V_i\}_{i=1}^\infty$ of compact intervals with 
$$\lim_{i\rightarrow \infty} \text{diam} (U_i)=\lim_{i\rightarrow \infty}\text{diam}(V_i)=0,$$
and such that for any $u\in U_n$ and $v\in V_n$ and any $j\leq n$ we get,
\begin{equation} \label{2d5.2/3}
\frac{1}{k_j}\xi(u,v,k_j,t_j)\leq F_{u_j,v_j}(t_j)+\frac{1}{j}
\end{equation} 
and
\begin{equation}\label{2d5.2/4}
\frac{1}{q_j}\xi(u,v,q_j,t_j)\geq F_{u_j,v_j}^\ast(t_j)-\frac{1}{j}.
\end{equation}
Take $U_1=U$, $V_1=V$, $k_1=1$, $q_1=2$ and assume that $U_n$, $V_n$, $k_n$ and $q_n$ have been defined such that $f^j(U_n)\cap \omega_1$ and $f^j(V_n)\cap \omega_2$ are infinite whenever $j$ is sufficiently large. Since $\left.f\right|_{\omega_1}$ is strongly transitive on $U$ and $\left.f\right|_{\omega_2}$ is strongly transitive on $V$, there is $s>q_n$ such that $u_{n+1}\in f^s(U_n)$ and $v_{n+1}\in f^s(V_n)$. Let $a\in U_n$ and $b\in V_n$ be such that $f^s(a)=u_{n+1}$ and $f^s(b)=v_{n+1}$. Then it is easy to see that $F_{ab}=F_{u_{n+1}v_{n+1}}$ and $F_{ab}^\ast=F_{u_{n+1}v_{n+1}}^\ast$. And the existence of $U_{n+1}\subset U_n$, $V_{n+1}\subset V_n$, $k_{n+1}>q_n$ and $q_{n+1}>k_{n+1}$ follows from Lemma~\ref{V51}. Note that $a,\,b,\,U_{n+1},\,V_{n+1}$ can be chosen such that $f^s(U_{n+1})\cap \omega_1$ resp. $f^s(V_{n+1})\cap \omega_2$ are infinite, see Lemma~\ref{PerfB}.

Take $x\in\bigcap_{j=1}^\infty U_j$ and $y\in\bigcap_{j=1}^\infty V_j$. Then for any $t\in T$ and any $(u,v)\in Q$ take $j$ such that $t=t_j$, $u=u_j$ and $v=v_j$. Since $x\in U_j$ and $y\in V_j$,  (\ref{2d5.2/3}) holds with $u=x$ and $v=y$. From the fact that $j$ can be arbitrarily large we have $F_{xy}(t)\leq F_{uv}(t)$. Therefore (\ref{2d5.2/1}) is true for any $t\in T$, and since $T$ is dense in $I$, for any $t$. The argument for (\ref{2d5.2/2}) is analogous.

It remains to show that $x$ can be chosen in $\omega_1$ and $y$ in $\omega_2$. Let $w\in \omega_1 \cap U$ be such that $\omega_f(w)=\omega_1$, and let $\{W_i\}_{i=1}^\infty$ be a decreasing sequence of compact neighborhoods of $w$ with $\lim_{i\rightarrow \infty}W_i=w$. Now apply Lemma~\ref{IL} we obtain that $\omega_f(x)\subset\omega_1$ and similarly $\omega_f(y)\subset \omega_2$. By Lemma~\ref{MaxBaz} there is $n\in \mathbb{N}$ such that $f^n(x)\in \omega_1$ and $f^n(y)\in \omega_2$. It is easy to see that $n$ can be chosen such that $f^n(x)\in U$ and $f^n(y)\in V$ and it is straightforward that (\ref{2d5.2/1}) and (\ref{2d5.2/2}) remain valid.

Proof of the Case b) is analogous to the previous proof, with exception that  for any $j$ $v_j=v=y=c$. 
\end{proof}

\begin{theorem}\label{BF}
Let $f \in \M(I)$ and let $\Omega=\{\omega_i\}_{i=1}^M$ be the set of maximal basic $\omega$-limit sets. For any $i,j$ from $\{1,\dots, M  \}$, set $G_{ij}=\inf \{F_{uv}| u\in \omega_i \mbox{ and } v\in \omega_j \}$. Then\\
$1)$ Each $G_{ij}$ is zero  on an interval $[0, \epsilon(i,j)]$, where $\epsilon(i,j)$ is a positive number.\\
$2)$ The set $\{G_{ij}| \omega_i \cap \omega_j \neq \emptyset  \}$ has a finite number of minimal elements.\\
\end{theorem}
\begin{proof}
$1)$ By Lemma~\ref{PerP} there are periodic points $p,q$ such that $p\in\omega_i$, $q\in\omega_j$ and $\min_i \delta_{pq}(i)=\epsilon>0$ and therefore $F_{pq}(t)=0$ for $t\in [0,\epsilon]$. Let $\epsilon(i,j)=\epsilon$, $G_{ij}\leq F_{pq}$ gives us the result.\\
$2)$ We will solve the problem by dividing it into three possible situations.\\
$a)$ $u\in C(f)\cap \omega_i$, $v\in C(f)\cap \omega_j$ without loss of generality we can assume that $u\in C_0(f)\cap \omega_i$ and $v\in C_0(f)\cap \omega_j$. Since we have a finite number of critical points the set $\{F_{uv} |u\in C_0(f)\cap \omega_i \mbox{ and } v\in C_0(f)\cap \omega_j   \}$ is also finite.\\
$b)$ $u\in C(f)\cap \omega_i$, $v\in C(f)^c \cap \omega_j$, we can assume that $u\in C_0(f)\cap \omega_i$. \\
Since $\omega_j$ is maximal basic set we have $f^{\ast}$-periodic set $U$ with $\omega_j \subset Orb(U)$ by corollary~\ref{cor} and by Lemma~\ref{PreStTr} $f|_{\omega_j}$ is strongly transitive in $U$. 
By Lemma~\ref{V52} for $u\in C(f)\cap \omega_i$ and any $v \in C(f)^c \cap \omega_j$ there is $x\in \omega_j$ with $ F_{ux}\leq F_{uv}$, together with a finite number of critical points we get that the set of minimal elements of the set of distributional functions of such points is finite.\\
$c)$  $u\in C(f)^c \cap \omega_i$, $v\in C(f)^c \cap \omega_j$, both $\omega_i$ and $\omega_j$ are maximal basic $\omega$-limit sets, we can find $f^{\ast}$-periodic sets $U,V$ with $\omega_i \subset Orb(U)$ $\omega_j \subset Orb(V)$, $f|_{\omega_i}$ is strongly transitive in $U$ and $f|_{\omega_j}$ is strongly transitive in $V$. We use Lemma~\ref{V52} and we have points $x\in \omega_i \cap U$ and $y\in \omega_j \cap V$ such that $F_{xy} \leq F_{uv}$. Therefore for one pair of maximal basic omega-limit sets we have minimal distributional function, together with finite number of maximal basic omega-limit set we have finite number of minimal elements. 
\end{proof}

\begin{corollary}
The set $G_{ii}$ has a finite number of minimal elements.
\end{corollary}

\section{Main result}

\begin{theorem}
Let $f \in \M(I)$, then both the Spectrum $\Sigma(f) $ and the Weak Spectrum $\Sigma_w(f)$ are nonempty and finite.
\end{theorem}
\begin{proof}
To prove the first part, let $D=\{F_{uv}| u \mbox{ and } v \mbox{ are isotetic}  \}$ and $E=\{F_{uv}| u,v \in \omega_i, \omega_i \in \Omega \}$. Clearly $E \subset D$, let $F_{uv} \in D$. Both $u \mbox{ and } v $ are isotetic,  therefore there is some maximal basic omega-limit set $\omega \in \Omega$ such that $\omega_{f^n}(u)\subset \omega$ and $\omega_{f^n}(v)\subset \omega$ for every positive integer $n$. From this we immediately get $\omega(u)\subset \omega$ and $\omega(v)\subset \omega$. We apply Lemma~\ref{MaxBaz} and we find $M\in \N$ such that $f^m (u), f^m(v) \in \omega$ for any $m\geq M$, let $u_1 = f^M(u) \mbox{ and } v_1 = f^M(v)$. From the definition of Distribution function we can see that $F_{uv}= F_{u_1 v_1} $, therefore $F_{uv} \in E$ and $D=E$. Corollary of Theorem~\ref{BF} gives us finite Spectrum $\Sigma(f)$.\\
To prove the second part, let $D_w=\{ F_{uv}| u \in C(f)^c, v \in I \mbox{ such that } \\
 \liminf_{n\to \infty} \delta_{uv}(n)=0  \}$ and $E_w=\{F_{uv}| u\in (\omega_i \cap C(f)^c) , v\in \omega_j, \\ \omega_i\cap \omega_j \neq \emptyset, \mbox{ where } \omega_i, \omega_j \mbox{ are maximal $\omega$-limit sets}  \}$.
Let $F_{uv} \in D_w$, $\omega(u) \subset \omega_i$, since $u\in C(f)^c$, where $\omega_i \in \Omega$. If $v\in C(f)$, then $\omega(v) $ is a cycle. From $ \liminf_{n\to \infty} \delta_{uv}(n)=0$ we get that $\omega(v) \cap \omega_i \neq \emptyset$, see Lemma~\ref{FG}, and therefore $F_{uv} \in  E_w$. If $v\in C(f)^c$, then $\omega(v) \subset \omega_j$, where $\omega_j \in \Omega$. If $\omega_i \neq \omega_j$, then, since $ \liminf_{n\to \infty} \delta_{uv}(n)=0$, there is a critical point $c$ in the intersection of $\omega_i$ and $\omega_j$, see Lemma~\ref{FG}, and $F_{uv}\in E_w$. Therefore $D_w \subset E_w$.
Conversely, let $F_{uv}\in E_w$. Let $u\in \omega_i$, $v\in \omega_j$ and $w\in \omega_i \cap \omega_j$.\\
Case $1)$. Let $u \in C(f)^c$ and $v\in C_0(f)$. Let fix $v$ and let $u\in \omega_i\in \Omega$, we apply Lemma~\ref{V52} and we find $x \in \omega_i$ such that $F_{xv}\leq F_{xv}$  and $F^*_{xv}=\chi_{(0,\infty)}$. Therefore $F_{xv}\in D_w$.\\
Case $2)$. If both $u \mbox{ and } v $ are in $C(f)^c$, then we can take $Q=\{(u,v),(w,w) \}$ and apply Lemma~\ref{V52} to get $x,y$ such that $F_{xy} \leq F_{uv}$ and $F_{xy}^{*}=\chi_{(0,\infty)}$. Hence $\liminf_{n\to \infty} \delta_{xy}(n)=0$ . Which implies $F_{xy} \in D_w$. \\
Thus we proved $D_w\subset E_w$ and if at least one of $u,v$ is not in $C_0(f)$, then $E_w$ has lower bound in $D_w$, if not (i.e. both $u,v$ are in $C_0(f)$), then we have only a finitely many $F_{uv}$. This argument together with Theorem~\ref{BF} gives us a finite Weak Spectrum $\Sigma_w(f)$.
\end{proof}

\section{Case with infinite spectrum}

In this section we present example of piecewise monotonic map without generator and with infinite spectrum. \\

Let $f:I \to I$, where $I=[0,1]$. $f$ with Markov condition, but without generator.\\
Let $J_1=[0,1/3],J_2=(1/3,2/3], J_3=(2/3,1]$, let

$   f(x) = \left\{
     \begin{array}{lr}
       f_1(x)=\frac{2}{3} - x &  x\in J_1 \\
       f_2(x)= \frac{4}{3} - x & x\in J_2 \\
       f_3(x)= x-\frac{2}{3} &  x\in J_3 \\
     \end{array}
   \right.$\\
     
 \begin{figure}[h]
 \centering
\begin{tikzpicture}

\draw[->, thick] (-0.3, 0.0) -- (3.6, 0.0) node[below, pos=3.05/3.6]{1} node[below, pos=1/3]{1/3} node[below, pos=17/27]{2/3} ;
\draw[->, thick] (0.0, -0.3) -- (0.0, 3.6) node[left, pos=0]{0} node[left, pos=3.05/3.6]{1} node[left, pos=1/3]{1/3} node[left, pos=17/27]{2/3} ;
\draw[] (0.0, 1.0) -- (3.0, 1.0) ;
\draw[] (0.0, 2.0) -- (3.0, 2.0);
\draw[] (0.0, 3.0) -- (3.0, 3.0);
\draw[] (1.0, 0.0) -- (1.0, 3.0);
\draw[] (2.0, 0.0) -- (2.0, 3.0);
\draw[] (3.0, 0.0) -- (3.0, 3.0);
\draw[ultra thick] (0.0, 2.0) -- (1.0, 1.0);
\draw[ultra thick] (1.0, 3.0) -- (2.0, 2.0);
\draw[ultra thick] (2.0, 0.0) -- (3.0, 1.0);

\end{tikzpicture}
\caption{$f(x)$}
\end{figure}

We can see that for any point $ x$ in $\inte(J_i)$ holds $f^3(x)=x  $.
   
Let $n$ is a positive integer, let $x_n=\frac{4n-1}{18n}$, $y_n = f_1(x_n)= \frac{8n+1}{18n}$, $z_n=f_2(y_n)=\frac{16n-1}{18n}$.\\
Let additionally $a_n=y_n - x_n = \frac{2n+1}{9n}$ \\
and $b_n = z_n - y_n= \frac{4n-1}{9n}$. \\
We can see that for any $n>1$ $a_n < b_n$ .\\

Distributional function of points $x_n$ and $y_n$ is\\   
    $F_{x_n,y_n}(t)=\left\{
     \begin{array}{lr}
        0 &  0\leq t < a_n \\
       1/3  & a_n \leq t < b_n \\
       2/3 &  b_n \leq t < a_n + b_n \\
       1 & a_n + b_n \leq t
     \end{array}
   \right. $

Since $a_n$ is decreasing sequence and $b_n$ is increasing we can see that for any $n<m$: $a_m<a_n$ and $b_n<b_m$,   
$F_{x_n,y_n}(t)<F_{x_m,y_m}(t) \mbox{ for } t\in (a_m,a_n)$ and $F_{x_m,y_m}(t)<F_{x_n,y_n}(t) \mbox{ for } t\in (b_n,b_m)$ and therefore incomparable. In the figure 2 we can see situation for $F_{x_2,y_2}$ and $F_{x_{2000},y_{2000}}$

\begin{figure}[h]
\centering
\begin{tikzpicture}

\draw[->, thick] (-0.3, 0.0) -- (3.6, 0.0)node[below, pos=3.05/3.6]{1} ;
\draw[->, thick] (0.0, -0.3) -- (0.0, 3.6)node[left, pos=0]{0} node[left, pos=3.05/3.6]{1};
\draw[thick] (0.0,0.0) -- (15/18, 0.0);
\draw[thick] (15/18, 1.0) -- (21/18, 1.0);
\draw[thick] (21/18, 2.0) -- (2.0, 2.0);
\draw[thick] (2.0, 3.0) -- (3.6, 3.0)node[below, pos=5/8]{$F_{x_2,y_2}$};
\draw[dashed] (15/18, 0.0) -- (15/18, 1.0);
\draw[dashed] (21/18, 0.0) -- (21/18, 2.0);
\draw[dashed] (2.0, 0.0) -- (2.0, 3.0) ;

\end{tikzpicture}
\begin{tikzpicture}

\draw[->, thick] (-0.3, 0.0) -- (3.6, 0.0)node[below, pos=3.05/3.6]{1} ;
\draw[->, thick] (0.0, -0.3) -- (0.0, 3.6)node[left, pos=0]{0} node[left, pos=3.05/3.6]{1};
\draw[thick] (0.0, 0.0) -- (0.6668, 0.0);
\draw[thick] (0.6668, 1.0) -- (1.3332, 1.0);
\draw[thick] (1.3332, 2.0) -- (2.0, 2.0);
\draw[thick] (2.0, 3.0) -- (3.6, 3.0)node[below, pos=5/8]{$F_{x_{2000},y_{2000}}$};
\draw[dashed] (15/18, 0.0) -- (15/18, 1.0);
\draw[dashed] (21/18, 0.0) -- (21/18, 2.0);
\draw[dashed] (2.0, 0.0) -- (2.0, 3.0);
\end{tikzpicture}
\caption{Lower distributional functions $F_{x_2, y_2}$, $F_{x_{2000},y_{2000}}$}
\end{figure}
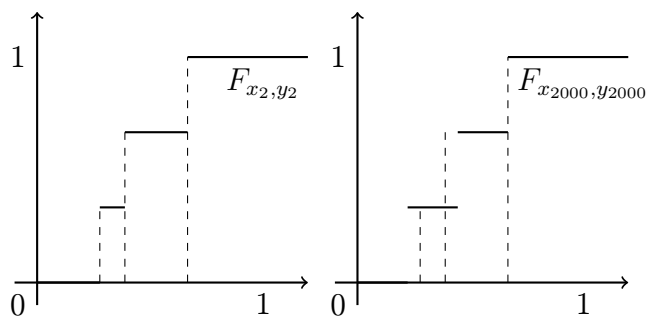

To see that $f(x)$ does not have generator let $K=[a,b] \subset \inte(J_1)$, then $f^{3m}(K)=K$ for any positive integer $m$.

\section*{Acknowledgments}
Research was funded by institutional support for the development of research
organization (I\v{C}47813059) and by SGS~18/2019 .

\end{document}